\documentclass[11pt]{article}
\usepackage{tikz}

\usepackage{empheq}
\usepackage[shortlabels]{enumitem}
\setlist[enumerate]{nosep}
\usepackage[doc]{optional}
\usepackage{xcolor}
\usepackage[colorlinks=true,
            linkcolor=refkey,
            urlcolor=lblue,
            citecolor=red]{hyperref}
\usepackage{float}
\usepackage{soul}
\usepackage{graphicx}
\definecolor{labelkey}{rgb}{0,0.08,0.45}
\definecolor{refkey}{rgb}{0,0.6,0.0}
\definecolor{Brown}{rgb}{0.45,0.0,0.05}
\definecolor{lime}{rgb}{0.00,0.8,0.0}
\definecolor{lblue}{rgb}{0.5,0.5,0.99}

 \usepackage{mathpazo}

\usepackage{stmaryrd}
\usepackage{amssymb}
\newcommand{\seppfive}{\setlength{\itemsep}{-5pt}}

\newcommand{\bId}{\ensuremath{\operatorname{\mathbf{Id}}}}
\newcommand{\bM}{\ensuremath{{\mathbf{M}}}}
\newcommand{\bb}{\ensuremath{\mathbf{b}}}
\newcommand{\bx}{\ensuremath{\mathbf{x}}}
\newcommand{\bzero}{\ensuremath{{\boldsymbol{0}}}}

\providecommand{\siff}{\Leftrightarrow}

\newcommand{\nnn}{\ensuremath{{n\in{\mathbb N}}}}

\newcommand{\menge}[2]{\big\{{#1}~\big |~{#2}\big\}}

\newcommand{\To}{\ensuremath{\rightrightarrows}}

\newcommand{\fenv}[1]%
{\ensuremath{\,\overrightarrow{\operatorname{env}}_{#1}}}
\newcommand{\benv}[1]%
{\ensuremath{\,\overleftarrow{\operatorname{env}}_{#1}}}

\newcommand{\scal}[2]{\left\langle{#1},{#2}  \right\rangle}

\newcommand{\RR}{\ensuremath{\mathbb R}}

\newcommand{\dom}{\ensuremath{\operatorname{dom}}}

\newcommand{\ran}{\ensuremath{\operatorname{ran}}}

\newcommand{\Id}{\ensuremath{\operatorname{Id}}}

\providecommand{\vcomp}{v_{T_m\cdots T_2 T_1}}

\newcommand{\pinf}{\ensuremath{+\infty}}

\newcommand{\bX}{\ensuremath{{\mathbf{X}}}}

\newcommand{\bR}{\ensuremath{{\mathbf{R}}}}
\newcommand{\bT}{\ensuremath{{\mathbf{T}}}}

\newcommand{\bD}{\ensuremath{{\boldsymbol{\Delta}}}}

\newcommand{\bA}{\ensuremath{{\mathbf{A}}}}

\newcommand{\bc}{\ensuremath{\mathbf{c}}}
\newcommand{\by}{\ensuremath{\mathbf{y}}}

{\begin{list}{}{%
\settowidth{\labelwidth}{\textrm{#1~}}%
\setlength{\leftmargin}{\labelwidth+\labelsep}}}
{\end{list}}
\usepackage{amsthm}
\usepackage[capitalize,nameinlink]{cleveref}
\crefname{equation}{}{equations}
\crefname{chapter}{Appendix}{chapters}
\crefname{item}{}{items}
\crefname{enumi}{}{}
\newtheorem{theorem}{Theorem}[section]

\newtheorem{lem}[theorem]{Lemma}

\newtheorem{cor}[theorem]{Corollary}

\newtheorem{prop}[theorem]{Proposition}

\newtheorem{thm}[theorem]{Theorem}
\newtheorem{example}[theorem]{Example}
\newtheorem{ex}[theorem]{Example}

\providecommand{\abs}[1]{\lvert#1\rvert}
\providecommand{\norm}[1]{\lVert#1\rVert}

\providecommand{\RA}{\Rightarrow}

\providecommand{\eps}{\varepsilon}

\newcommand{\bv}{\ensuremath{{\mathbf{v}}}}

\providecommand{\lam}{\lambda}
\providecommand{\RR}{\mathbb{R}}

\providecommand{\ran}{\operatorname{ran}}

\providecommand{\dom}{\operatorname{dom}}

\newcommand{\fix}{\ensuremath{\operatorname{Fix}}}

\providecommand{\gr}{\operatorname{gra}}
\providecommand{\gra}{\operatorname{gra}}
\providecommand{\Id}{\operatorname{{ Id}}}

\providecommand{\fady}{\varnothing}

\providecommand{\To}{\rightrightarrows}

\providecommand{\gr}{\operatorname{gra}}
\providecommand{\fix}{\operatorname{Fix}}
\providecommand{\ran}{\operatorname{ran}}

\providecommand{\Id}{\operatorname{Id}}

\providecommand{\fady}{\varnothing}

\providecommand{\RR}{\mathbb{R}}

\definecolor{myblue}{rgb}{.8, .8, 1}
  \newcommand*\mybluebox[1]{%
    \colorbox{myblue}{\hspace{1em}#1\hspace{1em}}}

\allowdisplaybreaks 

\begin{document}

%

\author{
Heinz H.\ Bauschke\thanks{
Mathematics, University
of British Columbia,
Kelowna, B.C.\ V1V~1V7, Canada. E-mail:
\texttt{heinz.bauschke@ubc.ca}.}
~and~ Walaa M.\ Moursi\thanks{
Simons Institute for the Theory of Computing, UC Berkeley,
Melvin Calvin Laboratory, \#2190, Berkeley, CA 94720, USA
and 
Mansoura University, Faculty of Science, Mathematics Department, 
Mansoura 35516, Egypt. 
E-mail: \texttt{walaa.moursi@gmail.com}.}
}

\title{ \textsc
The magnitude of the minimal displacement vector \\
for compositions and 
convex combinations\\ of firmly nonexpansive mappings}


\date{December 1, 2017}

\maketitle

\begin{abstract}
\noindent
Maximally monotone operators and firmly nonexpansive mappings play key
roles in modern optimization and nonlinear analysis. 
Five years ago, it was shown that if finitely many firmly
nonexpansive operators are all asymptotically regular (i.e., the
have or ``almost have'' fixed points), then the same is true for
compositions and convex combinations.

In this paper, we derive bounds on the magnitude of 
the minimal displacement
vectors of compositions and of convex combinations in terms of
the displacement vectors of the underlying operators. Our results
completely generalize earlier works. Moreover, we present various
examples illustrating that our bounds are sharp. 
\end{abstract}
{ 

}

\section{Introduction and Standing Assumptions}

Throughout this paper,

\begin{empheq}[box=\mybluebox]{equation}
\text{$X$ is
a real Hilbert space with inner
product $\scal{\cdot}{\cdot}$ }
\end{empheq}
and induced norm $\|\cdot\|$.
Recall that $T\colon X\to X$ is 
\emph{firmly nonexpansive} (see, e.g., \cite{BC2017}, \cite{GK}, and
\cite{GR} for further information)
if $(\forall (x,y)\in X\times X)$ 
$\|Tx-Ty\|^2\leq \scal{x-y}{Tx-Ty}$ and
that a set-valued operator $A\colon X\To X$ is 
\emph{maximally monotone} if it is \emph{monotone}, i.e., 
$\{(x,x^*),(y,y^*)\}\subseteq \gra A
\RA
\scal{x-y}{x^*-y^*}\geq 0$ and if the graph of $A$ cannot be properly
enlarged without destroying monotonicity\footnote{
We shall write $\dom A = \menge{x\in X}{Ax\neq\varnothing}$
for the \emph{domain} of $A$, $\ran A = A(X) = \bigcup_{x\in X}Ax$ for
the \emph{range} of $A$, and 
$\gr A=\menge{(x,u)\in X\times X}{u\in Ax}$ 
for the \emph{graph} of
$A$.}. 
These notions are equivalent (see \cite{Minty} and \cite{EckBer}) 
in the sense that 
if $A$ is maximally monotone, then
its \emph{resolvent} $J_A := (\Id+A)^{-1}$ is
firmly nonexpansive, and if $T$ is firmly nonexpansive,
then $T^{-1}-\Id$ is maximally monotone\footnote{
Here and elsewhere, $\Id$ denotes the \emph{identity} operator on $X$.}.

In optimization, one main problem is to find zeros of (sums of) 
maximally monotone
operators --- these zeros may correspond to critical points or solutions to
optimization problems. In terms of resolvents, the corresponding problem is
that of finding fixed points. 
For background material in fixed point theory and monotone operator
theory, we refer the reader to 
\cite{BC2017},
\cite{BorVanBook}, 
\cite{Brezis},
\cite{BurIus},
\cite{GK},
\cite{GR}, 
\cite{Rock70},
\cite{Rock98},
\cite{Simons1},
\cite{Simons2},
\cite{Zalinescu},
\cite{Zeidler2a},
\cite{Zeidler2b},
and \cite{Zeidler1}.
However, not every problem has a solution; equivalently, not
every resolvent has a fixed point.
To make this concrete, 
let us assume that $T\colon X\to X$ is firmly nonexpansive.
The deviation from $T$ possessing a fixed point is captured by the
notion of the \emph{minimal (negative) displacement vector}
which is well defined by\footnote{Given a nonempty closed convex
subset $C$ of $X$, we denote its \emph{projection mapping} or projector
by $P_C$.}
\begin{empheq}[box=\mybluebox]{equation}
v_T := 
P_{\overline{\ran}(\Id-T)}(0). 
\end{empheq}
If $T$ ``almost'' has a fixed point in the sense that 
$v_T = 0$, i.e., $0\in\overline{\ran}{(\Id-T)}$,  
then we say that $T$ is \emph{asymptotically regular}. 
From now on, we assume that 
\begin{empheq}[box=\mybluebox]{equation*}
\text{$I := \{1,2,\ldots,m\}$, where $m\in\{2,3,4,\ldots\}$ }
\end{empheq}
and that 
we are given $m$ firmly nonexpansive operators $T_1,\ldots,T_m$;
equivalently, $m$ resolvents of maximally monotone operators
$A_1,\ldots,A_m$:
\begin{empheq}[box=\mybluebox]{equation*}
(\forall i\in I)\quad T_i = J_{A_i} = (\Id+A_i)^{-1}
\;\;\text{is firmly nonexpansive,}
\end{empheq}
and we abbreviate the corresponding 
minimal displacement vectors by 
\begin{empheq}[box=\mybluebox]{equation}
\label{eq:assum:in}
(\forall i\in I)\quad v_i:= v_{T_i} =
P_{\overline{\ran}(\Id-T_i)}(0). 
\end{empheq}

A natural question is the following: 
\emph{What can be said about the minimal displacement vector of
$T$ when $T$ is either a composition or a convex combination of
$T_1,\ldots,T_n$?}

Five years ago, the authors of \cite{BMMW12} proved the
following: 
\begin{quotation}
\noindent
{\em 
If each $T_i$ is asymptotically regular, then so are
the corresponding compositions and convex combinations.
}
\end{quotation}
This can be expressed equivalently as
\begin{equation}
(\forall i\in I)\;\; v_i = 0
\quad\Rightarrow\quad v_T=0,
\end{equation}
where $T$ is either a composition or a convex combination of the
family $(T_i)_{i\in I}$. 
It is noteworthy that these results have been studied 
recently by Kohlenbach \cite{KGA17} and \cite{Kohlen17} 
from the viewpoint of ``proof mining''. 

In this work, we 
obtain \emph{sharp bounds} on the magnitude of the minimal
displacement vector of $T$ that hold true 
\emph{without any assumption of asymptotic regularity of the
given operators.}
The proofs rely on techniques that are new and that were 
introduced in 
\cite{BMMW12} and \cite{Bau03} (where projectors were
considered). 
The new results concerning compositions are presented in
Section~\ref{s:compo} while convex combinations are dealt with in
Section~\ref{s:convco}. Finally, 
our notation is standard and follows \cite{BC2017} to which we
also refer for standard facts not mentioned here. 

\section{Compositions}
\label{s:compo}
In this section, 
 we explore compositions. 

\begin{prop}
\label{t:main}
\label{prop:comp:v:eps}
$(\forall \eps>0)$
$(\exists x\in X)$
such that 
$\norm{x-T_mT_{m-1}\cdots T_1x}\le \eps +\sum_{k=1}^{m}\norm{v_k}$.
\end{prop}
\begin{proof}
The proof is broken up into several steps.
Set 
\begin{equation}
 \label{eq:A:tilde}
(\forall i\in I)\quad \widetilde{A}_i:=-v_i+A_i(\cdot-v_i).
\end{equation}
and observe that \cite[Proposition~23.17(ii)\&(iii)]{BC2017}
yields
\begin{equation}
\label{eq:res:op}
(\forall i\in I)\quad \widetilde{T}_i:=J_{ \widetilde{A}_i}
=v_i+J_{A_i}=v_i+T_i.
\end{equation}
We also work in 
\begin{equation}
\bX := X^m = \menge{\bx = (x_i)_{i\in I}}{(\forall i\in I)\;
x_i\in X}, \quad\text{with}\;\;
\scal{\bx}{\by} = \sum_{i\in I}\scal{x_i}{y_i}, 
\end{equation}
where we embed the original operators via
\begin{equation}
\bT\colon X^m\to X^m\colon (x_i)_{i\in I}\mapsto (T_ix_i)_{i\in I}
\;\;\text{and}\;\;
\bA\colon X^m\To X^m\colon (x_i)_{i\in I}\mapsto
\times(A_ix_i)_{i\in I}.
\end{equation}
Denoting the identity on $X^m$ by $\bId$, we observe that
\begin{equation}
\label{eq:def:bT}
J_{\bA} = (\bId + \bA)^{-1} = T_1\times\cdots \times T_m = \bT. 
\end{equation}
Because
${\ran}(\bId-\bT)=\ran(\Id-T_1) \times \dots \times
\ran(\Id-T_m)$ and hence 
$\overline{\ran}(\bId-\bT)=\overline{\ran}(\Id-T_1) \times \dots
\times \overline{\ran}(\Id-T_m)$, we have 
(e.g., by using \cite[Proposition~29.3]{BC2017})
\begin{equation}
\label{eq:def:bv}
\bv:=(v_i)_{i\in I}=P_{\overline{\ran}(\bId-\bT)}\bzero.
\end{equation}
Finally, 
define the \emph{cyclic right-shift operator}
\begin{equation}
\bR\colon X^m\to X^m\colon (x_1,x_2,\ldots,x_m)\mapsto
(x_m,x_1,\ldots,x_{m-1}) 
\;\;\text{and}\;\;
\bM \coloneqq \bId-\bR, 
\end{equation}
and the \emph{diagonal} subspace 
\begin{equation}
\bD := \menge{\bx= (x)_{i\in I}}{x\in X},
\end{equation}
with orthogonal complement $\bD^\perp$. 

\textsc{Claim 1:} $\bv\in\overline{\ran}\ (\bA(\cdot-\bv)+\bM)$.\\
Indeed, \cref{eq:assum:in} implies that 
$(\forall i\in I)$ $v_i\in\overline{\ran }\ (\Id-T_i)
=\overline{\ran}\ (\Id-J_{A_i})
=\overline{\ran}\ J_{A_i^{-1}}
=\overline{\dom}\ (\Id +A_i^{-1})
=\overline{\dom}\ A_i^{-1}
=\overline{\ran}\ A_i
=\overline{\ran}\ A_i(\cdot-v_i)
$. 
Hence, $\bv\in \overline{\ran}\ \bA(\cdot-\bv)
=\overline{\ran \bA(\cdot-\bv)+\bzero}
\subseteq
\overline{\ran \bA(\cdot-\bv)+\bD^\perp}
$. 
On the other hand, we 
learn  from \cite[Corollary~2.6]{BMMW12}
(applied to $\bA(\cdot-\bv)$)
that 
$\overline{\ran}\ (\bA(\cdot-\bv)+\bM)=\overline{\ran\bA(\cdot-\bv)+\bD^\perp}$. 
Altogether, we obtain that 
$\bv\in\overline{\ran}\ (\bA(\cdot-\bv)+\bM)$ and
\textsc{Claim~1} is verified. 

\textsc{Claim~2:}
$(\forall\varepsilon>0)$
$(\exists (\bb,\bx)\in\bX\times \bX)$
$\|\bb\|\leq\varepsilon$
and 
$\bx=\bv+\bT(\bb+\bR\bx)$.\\
Fix $\varepsilon>0$. 
In view of \textsc{Claim~1}, 
there
exists $\bx\in\bX$ and $\bb\in\bX$ such that
$\|\bb\|\leq\varepsilon$ 
and $\bb\in -\bv+ \bA(\bx-\bv)+\bM\bx$. 
Hence, $\bb+\bR\bx=\bb+\bx-\bM \bx \in \bx+\bA(\bx-\bv)-\bv
=(\bId+(-\bv+\bA(\cdot-\bv))\bx$. Thus,
$\bx = J_{-\bv+\bA(\cdot-\bv)}(\bb+\bR\bx) = \bv+\bT(\bb+\bR\bx)$,
where the last identity follows from
\cref{eq:res:op},
\cref{eq:def:bT} and \cref{eq:def:bv}.

\textsc{Claim~3}:
$(\forall\varepsilon>0)$
$(\exists (\bc,\bx)\in\bX\times \bX)$
$\|\bc\|\leq\varepsilon$
and 
$\bx=\bc+\bv+ \bT(\bR\bx)$.\\
Fix $\varepsilon>0$, 
let $\bb$ and $\bx$ be as in \textsc{Claim~2}, 
and set 
$\bc := \bx-\bv-\bT(\bR\bx)=
\bT(\bb+\bR\bx) - \bT(\bR\bx)$.
Then, since $\bT$ is nonexpansive, 
$\|\bc\| = \|\bT(\bb+\bR\bx) - \bT(\bR\bx)\|
\leq \|\bb\|\leq\varepsilon$, and \textsc{Claim~3} thus holds.

\textsc{Conclusion:}\\
Let $\varepsilon > 0$.
By
\textsc{Claim~3} (applied to $\varepsilon/\sqrt{m}$),  there exists 
$(\bc,\bx)\in \bX\times \bX$
such that $\norm{\bc}\le \eps/\sqrt{m}$
and $\bx=\bc+\bv+ \bT(\bR\bx)$.
Hence $\sum_{i\in I} \|c_i\| \leq \|\bc\|\sqrt{m} \leq
\varepsilon$ and
$(\forall i\in I)$ 
$x_i=c_i+v_i+T_ix_{i-1}$,
where $x_0:=x_m$.
The triangle inequality and the nonexpansiveness
of each $T_i$ thus yields
\begin{align}
\norm{T_mT_{m-1}\cdots T_1 x_0-x_0}
&=\norm{T_mT_{m-1}\cdots T_1 x_0-x_m}
\nonumber
\\
&=\big\|T_mT_{m-1}\cdots T_2T_1 x_0-T_mT_{m-1}\cdots T_2x_1
\nonumber
\\
&\qquad 
+T_mT_{m-1}\cdots T_3T_2x_1-T_mT_{m-1}\cdots T_3x_2
\nonumber
\\
&\qquad 
+T_mT_{m-1}\cdots T_4T_3x_2-T_mT_{m-1}\cdots T_4x_3
\nonumber
\\
&\qquad +\cdots
\nonumber
\\
&\qquad 
+T_mT_{m-1}x_{m-2} -T_mx_{m-1}
\nonumber
\\
&\qquad 
+T_mx_{m-1}-x_m\big\|
\nonumber
\\
&\le 
\norm{T_mT_{m-1}\cdots T_2 T_1 x_0-T_mT_{m-1}\cdots T_2x_1}
\nonumber
\\
&\qquad 
+\norm{T_mT_{m-1}\cdots T_3T_2x_1-T_mT_{m-1}\cdots T_3x_2}
\nonumber
\\
&\qquad 
+\norm{T_mT_{m-1}\cdots T_4T_3x_2-T_mT_{m-1}\cdots T_4x_3}
\nonumber
\\
&\qquad +\cdots
\nonumber
\\
&\qquad 
+\norm{T_mT_{m-1}x_{m-2} -T_mx_{m-1}}
\nonumber
\\
&\qquad 
+\norm{T_mx_{m-1}-x_m}
\nonumber
\\
&\le 
\norm{T_1 x_0-x_1}
+\norm{T_2x_1-x_2}
+\norm{T_3x_2-x_3}
\nonumber
\\
&\qquad+\cdots
+\norm{T_{m-1}x_{m-2}-x_{m-1}}
+\norm{T_mx_{m-1}-x_m}
\nonumber
\\
&=\norm{c_1+v_1}+\norm{c_2+v_2}+\cdots+\norm{c_m+v_m}
\nonumber
\\
&\le \sum_{k=1}^{m}\norm{c_i}+ \sum_{k=1}^{m}\norm{v_i}\notag\\
&\le 
\eps+ \sum_{k=1}^{m}\norm{v_i},
\end{align}
as claimed. 
\end{proof}

We are now ready for our first main result. 

\begin{thm}
\label{cor:comp}
$\displaystyle \norm{\vcomp}\le
\norm{v_{T_1}}+\cdots+\norm{v_{T_m}}$. 
\end{thm}
\begin{proof}
By \cref{prop:comp:v:eps}, we have
$(\forall \eps>0)$ 
$\norm{\vcomp}\leq\eps+\norm{v_{T_1}}+\cdots+\norm{v_{T_m}}$ 
and the result thus follows.
\end{proof}

As an immediate consequence of \cref{cor:comp},
we obtain the first main result of \cite{BMMW12}:

\begin{cor}
{\rm \cite[Corollary~3.2]{BMMW12}}
\label{cor:Victoria:1:comp}
Suppose that $v_1=\cdots=v_m=0$.
Then $\vcomp=0$.
\end{cor}

We now show that 
the bound on $\norm{\vcomp}$ given in \cref{cor:comp}
is sharp:

\begin{example}
Suppose that $X=\RR$,
$T_1\colon X\to X\colon x\mapsto x-a_1$,
and 
$T_2\colon X\to X\colon x\mapsto x-a_2$,
where $(a_1,a_2) \in \RR\times \RR$.
Then  $(v_{T_1},v_{T_2},v_{T_2T_1})=(a_1,a_2,
a_1+a_2)$
and 
$\abs{a_1+a_2}=\abs{v_{T_2T_1}}\le\abs{v_1}+\abs{v_2}=\abs{a_1}+\abs{a_2}$;
moreover, 
the inequality is an 
equality if and only if $a_1a_2\geq 0$. 
\end{example}

\begin{proof}
On the one hand, it is clear that
$\ran (\Id-T_1)=\{a_1\}$
 and likewise
 $\ran (\Id-T_2)=\{a_2\}$.
 Consequently,
 $(v_1,v_2)=(a_1,a_2)$.
On the other hand,
 $T_2T_1\colon X\to X\colon x\mapsto x-a_1
 -a_2=x-(a_1+a_2)$,
therefore
 $\ran (\Id-T_2T_1)=\{a_1+a_2\}$.
 Hence, $v_{T_2T_1}=a_1+a_2$,
 $\abs{v_{T_2T_1}}=\abs{a_1+a_2}$  
 and 
 $\abs{v_1}+\abs{v_2}=\abs{a_1}+\abs{a_2}$,
  and the conclusion follows.
\end{proof}

The remaining results in this section concern the effect of
cyclically permuting 
the operators in the composition.

\begin{prop}
\label{subeq:1st}
$v_{T_mT_{m-1}\cdots  T_2T_1}
=v_{T_{m-1}T_{m-2}\cdots T_1T_{m}}
=\cdots
=v_{T_1T_{m}\cdots T_2}$.
\end{prop}
\begin{proof}
We start by proving that
if $S_1\colon X\to X$ and $S_2\colon X\to X$
are averaged\footnote{Let $S\colon X\to X$.
Then $S$ is $\alpha$-averaged if there exists
$\alpha\in [0,1[$ such that
$S=(1-\alpha)\Id+\alpha N$ and $N\colon X\to X$ 
is nonexpansive.}, 
then
\begin{equation}
\label{eq:eq:gaps}
v_{S_2S_1}=v_{S_1S_2}.
\end{equation}
To this end, let $x\in X$
and note that $S_2S_1$
 and $S_1S_2$ are $\alpha$-averaged 
where $\alpha\in [0,1[$ by, e.g., 
 \cite[Remark~4.34(iii)~and~Proposition~4.44]{BC2017}.
 Using \cite[Proposition~2.5(ii)]{Moursi17} 
 applied to 
 $S_2S_1$ and $S_1S_2$ yields
\begin{align}
\norm{v_{S_2S_1}-v_{S_1S_2}}^2
&\leftarrow \norm{(S_2S_1)^{n}x -(S_2S_1)^{n+1}x
-((S_1S_2)^{n}S_1x -(S_1S_2)^{n+1}S_1x)}^2
\nonumber\\
&=\norm{(S_2S_1)^{n}x -(S_2S_1)^{n+1}x-(S_1(S_2S_1)^{n}x 
-S_1(S_2S_1)^{n+1}x)   }^2
\nonumber\\
&=\norm{(\Id-S_1)(S_2S_1)^{n}x -(\Id-S_1)(S_2S_1)^{n+1}x }^2
\nonumber\\
&\le\tfrac{\alpha}{1-\alpha}( \norm{(S_2S_1)^{n}x -(S_2S_1)^{n+1}x}^2
-\norm{S_1(S_2S_1)^{n}x -S_1(S_2S_1)^{n+1}x}^2)
\nonumber\\
&\le \tfrac{\alpha}{1-\alpha}( \norm{(S_2S_1)^{n}x -(S_2S_1)^{n+1}x}^2
-\norm{(S_1S_2)^{n}S_1x -(S_1S_2)^{n+1}S_1x}^2)
\nonumber\\
&\to \tfrac{\alpha}{1-\alpha}( \norm{v_{S_2S_1}}^2-\norm{v_{S_1S_2}}^2)= 0,
\end{align}
where the last
identity follows from
\cite[Lemma~2.6]{Sicon2014}.
Because $T_{m-1}T_{m-2}\ldots T_1$ is averaged by
 \cite[Remark~4.34(iii)~and~Proposition~4.44]{BC2017}, 
we can and do apply \cref{eq:eq:gaps},
with $(S_1,S_2)$ replaced by $(T_{m-1}T_{m-2}\ldots T_1,T_m)$,
to deduce that 
$v_{T_mT_{m-1}\cdots  T_2T_1} =v_{T_{m-1}T_{m-2}\cdots
T_1T_{m}}$. 
The remaining identities follow similarly. 
\end{proof}
\begin{prop}
\label{prop:attain:2}
We have 
\begin{subequations}
\begin{align}
v_{T_mT_{m-1}\cdots T_1}\in \ran(\Id-T_mT_{m-1}\cdots T_1)
& \siff 
v_{T_{m-1}\cdots T_1T_m}\in \ran(\Id-T_{m-1}\cdots T_1T_m)
\label{eq:a}
\\
& \siff \cdots
\\
& \siff 
v_{T_1T_m\cdots T_2}\in \ran(\Id-T_1T_m\ldots T_2).
\end{align}
\end{subequations}
\end{prop}
\begin{proof}
We prove the implication 
``$\RA$'' of \cref{eq:a}:
Suppose that $(\exists y\in X)$
$v_{T_mT_{m-1}\cdots T_1}=y-T_mT_{m-1}\cdots T_1 y$, i.e., 
$y\in \fix(v_{T_m\cdots T_1}+T_m\cdots T_1)$.
By \cite[Proposition~2.5(iv)]{BM2015:AFF}, we have
$v_{T_mT\cdots T_1}=(T_m\cdots T_1)y-(T_m\cdots T_1)^2 y$.
Using 
\cref{subeq:1st}, we obtain 
\begin{align}
\norm{v_{T_{m-1}\cdots T_1T_m}}
&=\norm{v_{T_m\cdots T_2T_1}}
=\norm{(T_mT_{m-1}\cdots T_1)y-(T_mT_{m-1}\cdots T_1)^2 y}
\nonumber \\
&\le\norm{T_{m-1}\cdots T_1y-(T_{m-1}\cdots T_1T_m)T_{m-1}\cdots T_1 y}
\nonumber
\\
&\le \norm{y-T_mT_{m-1}\cdots T_1 y}=\norm{v_{T_mT_{m-1}\cdots T_1}}
=\norm{v_{T_{m-1}\cdots T_1T_m}}.
\end{align}
Consequently, 
$\norm{v_{T_{m-1}\cdots T_1T_m}}=\norm{T_{m-1}\cdots T_1
y-(T_{m-1}\cdots T_1T_m)T_{m-1}\cdots T_1 y}$
and hence
\begin{equation}
v_{T_{m-1}\ldots T_1T_m}
=T_{m-1}\cdots T_1 y-(T_{m-1}\cdots T_1T_m)T_{m-1}\ldots T_1 y
\in
\ran(\Id-T_{m-1}\ldots T_1T_m).
\end{equation}
The opposite implication and 
the remaining $m-2$ equivalences are proved similarly. 
\end{proof}

The following example, taken from De 
Pierro's \cite[Section~3~on~page~193]{DeP2000},
illustrates that
the conclusion of \cref{prop:attain:2}
does not necessarily hold if the operators
are permuted noncyclically. 

\begin{example}
\label{ex:attain:3}
Suppose that $X=\RR^2$,
$m=3$, 
$C_1=\RR\times\{0\}$,
$C_2=\RR\times\{1\}$,
$C_3=\menge{(x,y)\in \RR^2}{y\ge 1/x>0}$, and
$(T_1,T_2,T_3)=(P_{C_1},P_{C_2},P_{C_3})$.
Then, $v_{T_3T_2T_1}=v_{T_3T_1T_2}=0$,
$v_{T_3T_2T_1}\in \ran(\Id-T_3T_2T_1)$ 
but 
$v_{T_3T_1T_2}\not\in \ran(\Id-T_3T_1T_2)$.
\end{example}
\begin{proof}
Note that $T_2T_1=P_{C_2}P_{C_1}=P_{C_2}=T_2$
and 
$T_1T_2=P_{C_1}P_{C_2}=P_{C_1}=T_1$.
Consequently,
$(T_3T_2T_1,T_3T_1T_2)=(P_{C_3}P_{C_2},P_{C_3}P_{C_1})$.
The claim that 
$v_{T_3T_2T_1}=v_{T_3T_1T_2}=0$
follows from \cite[Theorem~3.1]{Bau03},
or \cref{cor:comp}
applied with $m=3$.
This and \cite[Lemma~2.2(i)]{BB94} imply that
$\fix T_3T_2T_1=\fix P_{C_3}P_{C_2}=C_3\cap C_2\neq \fady$,
whereas  
$\fix T_3T_1T_2=\fix P_{C_3}P_{C_1}=C_3\cap C_1= \fady$.
Hence, 
$v_{T_3T_2T_1}\in \ran(\Id-T_3T_2T_1)$ 
but 
$v_{T_3T_1T_2}\not\in \ran(\Id-T_3T_1T_2)$.
\end{proof}
\begin{figure}[H]
\begin{center}
\includegraphics[scale=0.45]{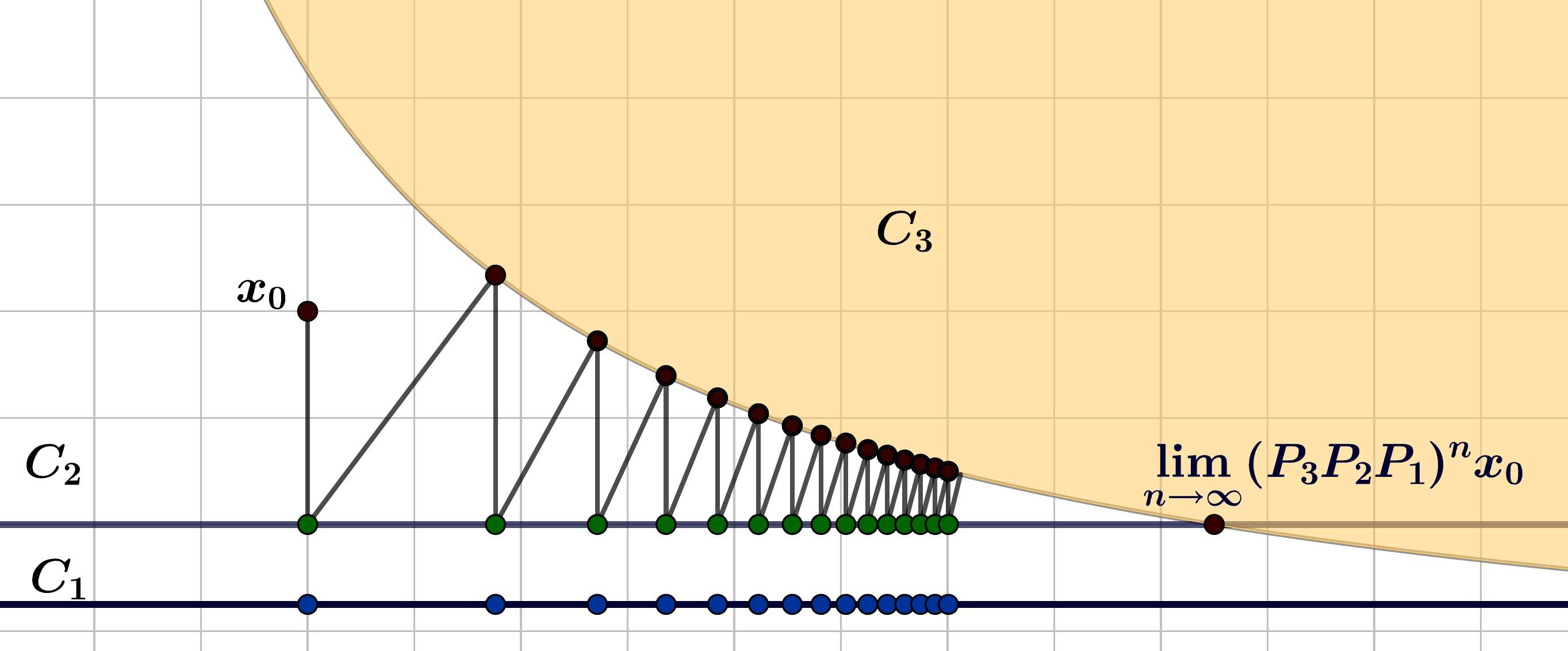}
\end{center}
\caption{A \url{GeoGebra} \cite{geogebra} snapshot that 
illustrates the behaviour of the sequence 
$((P_3P_2P_1)^nx_0)_\nnn$ in \cref{prop:attain:2}.
The first few iterates of the 
sequences 
$(P_1(P_3P_2P_1)^nx_0)_\nnn$ (blue points),
$(P_2P_1(P_3P_2P_1)^nx_0)_\nnn$ (green points),
 and 
 $((P_3P_2P_1)^nx_0)_\nnn$ (black points)
 are also depicted.
}
\end{figure}
\begin{figure}[H]
\begin{center}
\includegraphics[scale=0.45]{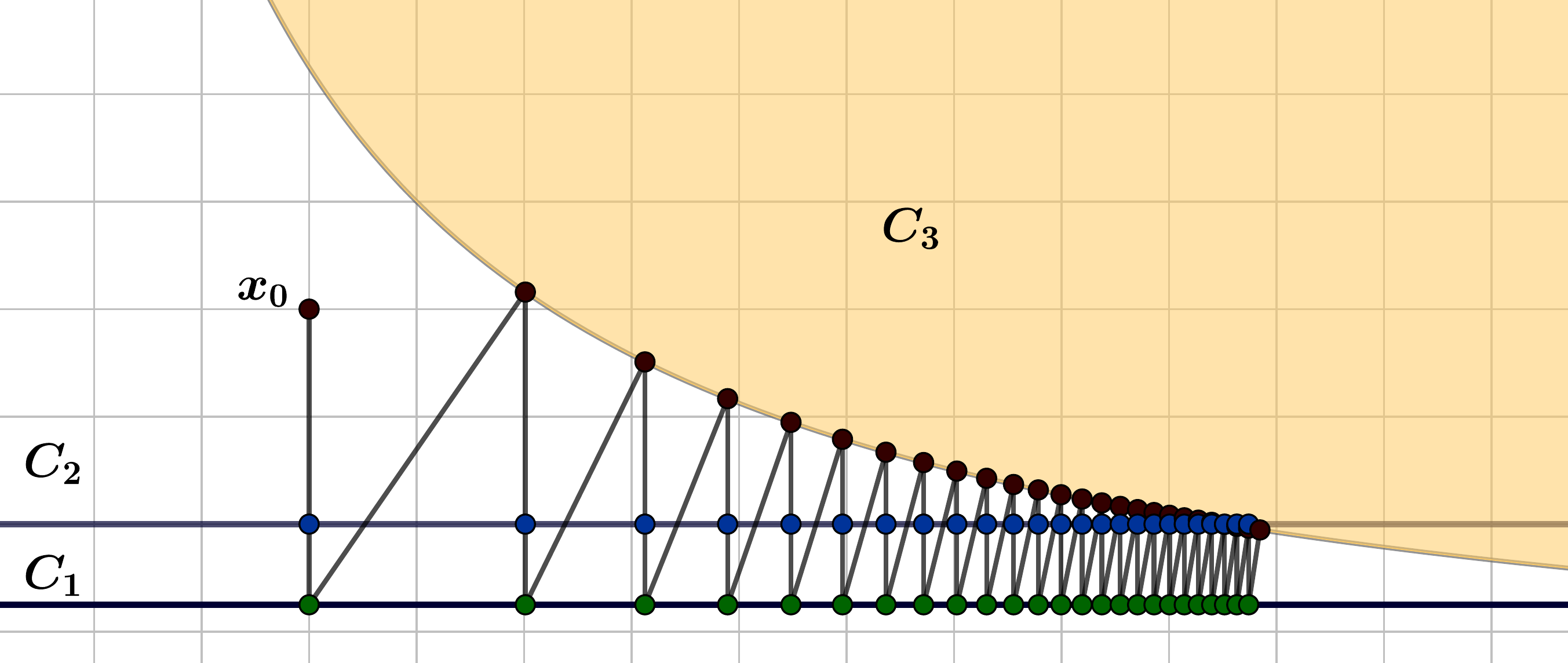}
\end{center}
\caption{A \url{GeoGebra} \cite{geogebra} snapshot that 
illustrates the behaviour of the sequence 
$((P_3P_1P_2)^nx_0)_\nnn$ in \cref{prop:attain:2}.
The first few iterates of the 
sequences 
$(P_1(P_3P_1P_2)^nx_0)_\nnn$ (green points),
$(P_2P_1(P_3P_1P_2)^nx_0)_\nnn$ (blue points),
 and 
 $((P_3P_1P_2)^nx_0)_\nnn$ (black points)
 are also depicted.
}
\end{figure}

\section{Convex Combinations}
\label{s:convco}

We start with the following useful lemma.

\begin{lem}
\label{lem:n:3*mono}
Suppose 
$(\forall i\in I)$ 
$A_i$ is $3^*$ monotone\footnote{We recall that
a monotone operator $B\colon X\To X$ is 
\emph{3* monotone} (see \cite{Br-H})
(this is also known as \emph{rectangular}) if 
$(\forall(x,y^*)\in\dom B\times\ran B)$
$\sup_{(z,z^*)\in\gr B} \scal{x-z}{z^*-y^*}<\pinf$.}  
and $\dom A_i=X$. 
Let $(\alpha_i)_{i\in I}$ be 
 a family of nonnegative real numbers.
Then the following hold:
\begin{enumerate}
 \item
 \label{lem:n:3*mono:i}
 $\sum_{i\in I}\alpha_i A_i$ is maximally monotone,
 $3^*$ monotone
and  $\dom\left(\sum_{i\in I}\alpha_i A_i\right)=X$.
 \item
  \label{lem:n:3*mono:ii}
 $\overline{\ran}(\sum_{i\in I}\alpha_i A_i)
 =\overline{\sum_{i\in I}\alpha_i\ran A_i} $.
\end{enumerate}
\end{lem}
\begin{proof}
Note that $(\forall i\in I)$, 
$\alpha_iA_i$ is
maximally monotone, 
$3^*$ monotone and 
$\dom \alpha_iA_i=X$.
  
\ref{lem:n:3*mono:i}:
The proof proceeds  by induction.
For $n=2$, 
the $3^*$ monotonicity of $\alpha_1A_1+\alpha_2A_2$
follows from \cite[Proposition~25.22(ii)]{BC2017}, 
whereas the maximal monotonicity of $\alpha_1A_1+\alpha_2A_2$
follows from, e.g., \cite[Proposition~25.5(i)]{BC2017}.
Now suppose that for some $n\ge 2$ it holds that
 $\sum_{i=1}^{n}\alpha_iA_i$ is maximally monotone and 
 $3^*$ monotone.
 Then $\sum_{i=1}^{n+1}\alpha_iA_i=\sum_{i=1}^{n}\alpha_iA_i
 +\alpha_{n+1}A_{n+1}$,
 which is maximally monotone
  and $3^*$ monotone, where the conclusion follows from
  applying the base case 
  with $(\alpha_1,\alpha_2,A_1,A_2)$ replaced by
  $(1,\alpha_{n+1},\sum_{i=1}^{n}\alpha_iA_i,A_{n+1})$.
  
\ref{lem:n:3*mono:ii}:
Combine \ref{lem:n:3*mono:i}
 and \cite[Corollary~6]{Penn01}.
\end{proof}

From this point onwards, let 
$(\lam_i)_{i\in I}$
be 
in $\left]0,1\right]$ with $\sum_{i\in I}\lam_i=1$, and set
\begin{empheq}[box=\mybluebox]{equation}
\overline{T} := \sum_{i\in I} \lambda_i T_i.
\end{empheq}

We are now ready for our second main result.

\begin{thm}
\label{cor:min:v}
$\norm{v_{\overline{T}}}
\le 
\norm{\sum_{i\in I}\lam_iv_{T_i}}$.
\end{thm}
\begin{proof}
It follows from  \cite[Examples~20.7~and~25.20]{BC2017}
that $(\forall i\in I)$
$\Id-T_i$ is maximally monotone, $3^*$ monotone
 and $\dom(\Id-T_i)=X$.
This and 
 \cref{lem:n:3*mono}\ref{lem:n:3*mono:ii}
 (applied with $(\alpha_i, A_i)$
 replaced by $(\lam_i,\Id-T_i)$
imply that
  \begin{equation}
  \label{prop:conv:comb}
\overline{\ran}\left(\Id-\overline{T}\right)
=\overline{\ran}\sum_{i\in I}\lam_i (\Id-T_i)
= \overline{\sum_{i\in I}\lam_i\ran(\Id-T_i)}.
\end{equation}
Now, 
on the one hand, it follows from the definition of 
$v_{\overline{T}}$ that
\begin{equation}
\label{eq:conv:bnd}
\left(\forall y\in \overline{\ran}\left(\Id-\overline{T}\right)\right)
\qquad\norm{ v_{\overline{T}}}\le \norm{y}.
\end{equation}
On the other hand, 
the definition of $v_i$
implies that
$(\forall i\in I)$
$v_i\in \overline{\ran}(\Id-T_i) $.
Hence,  
$\lam_iv_i\in \lam_i\ \overline{\ran}(\Id-T_i)$.
Therefore, 
$\sum_{i\in I}\lam_iv_i\in \sum_{i\in I}\lam_i\ \overline{\ran} (\Id-T_i)
\subseteq \overline{ \sum_{i\in I}\lam_i\ \ran (\Id-T_i)}
=\overline{\ran}\left(\Id-\overline{T}\right)$,
where the last identity follows from 
\cref{prop:conv:comb}. 
Now apply \cref{eq:conv:bnd}
with $y$ replaced by
$\sum_{i\in I}\lam_iv_{i}$.
\end{proof}

As an easy consequence of \cref{cor:min:v},
we obtain the second main result of \cite{BMMW12}:

\begin{cor}
\label{cor:Victoria:1:convcom}
{\rm \cite[Theorem~5.5]{BMMW12}}
Suppose that $v_1=\cdots=v_m=0$.
Then $v_{\overline{T}} 
=0$.
\end{cor}

 The bound we provided in \cref{cor:min:v}
is sharp as we illustrate now:

\begin{example}
\label{ex:convcom:bound}
Let $a\in X$ 
and suppose that 
$T\colon X \to X\colon x\mapsto x-a$.
Then $v_T=a$
 and therefore $\fix T\neq\fady \siff a= 0$.
Set $(\forall i\in I)$ $T_i=T$.
Then $\overline{T}=\sum_{i\in I}\lam_iT_i=T$,
 $(\forall i\in I)$
$v_i=v_{\overline{T}}=a$.
Consequently, 
$\norm{v_{\overline{T}}}=\norm{a}=\norm{\sum_{i\in I}\lam_ia}
=\norm{\sum_{i\in I}\lam_iv_i}$.
\end{example}

\cref{ex:convcom:bound}
suggests that 
the identity $v_{\overline{T}}=\sum_{i\in I}\lam_iv_i$ holds 
true; however, 
the following example provides a negative answer 
to this conjecture.

\begin{example}
Suppose that $m=2$, 
that $T_1\colon X\to X\colon x\mapsto x-a_1$, and that
$T_2\colon X\to X\colon x\mapsto \tfrac{1}{2}x-a_2$,
where $(a_1,a_2)\in (X\smallsetminus\{0\})\times X$. 
Then 
 $\ran (\Id-T_1)=\{a_1\}$,
  $\ran (\Id-T_2)=X$,
  $\ran (\Id-\overline{T})=X$, 
  and 
$0=v_{\overline{T}}\neq\lam_1 v_1+\lam_2 v_2=\lam_1 a_1$.
\end{example}  
\begin{proof}
On the one hand, one can easily verify that 
$(v_1,v_2)=(a_1,0)$; hence,
$\lam_1 v_1+\lam_2 v_2=\lam_1 a_1\neq 0$.
On the other hand, 
$\overline{T}\colon X\to X\colon x\mapsto 
\tfrac{\lam_1+1}{2}x-(\lam_1 a_1+\lam_2 a_2)$.
Hence, $\overline{T}$ is a Banach contraction,
and therefore, $\fix\overline{T}\neq \fady$.
Consequently, $v_{\overline{T}}=0$.  
\end{proof}

\section*{Acknowledgments}
The research of HHB was partially supported by a Discovery Grant
of the Natural Sciences and Engineering Research Council of
Canada. 
WMM was supported by the Simons Institute for the Theory
of Computing research fellowship.

\end{document}